\documentclass{amsart}
\usepackage{amsfonts, color}
\usepackage{latexsym}
\usepackage{amssymb}
\usepackage{amsmath}

\usepackage{amsthm}
\usepackage{epsfig}
\usepackage{graphicx,colortbl}
\usepackage{url,hyperref}
\usepackage{mathtools, color, changes}

\newcommand{\R}{\mathbb R}
\newcommand{\N}{\mathbb N}

\newtheorem{thm}{Theorem}[section]
\newtheorem{cor}{Corollary}[section]

\theoremstyle{remark}
\newtheorem*{rmk}{Remark}

\numberwithin{equation}{section}

\begin{document}


\title{Brunn-Minkowski inequality for $\theta$-convolution bodies via Ball's bodies}

\author[D.\,Alonso]{David Alonso-Guti\'errez}
\address{\'Area de an\'alisis matem\'atico, Departamento de matem\'aticas, Facultad de Ciencias, Universidad de Zaragoza, Pedro Cerbuna 12, 50009 Zaragoza (Spain), IUMA}
\email[(David Alonso)]{alonsod@unizar.es}

\author{Javier Mart\'in Go\~ni}
\address{\'Area de an\'alisis matem\'atico, Departamento de matem\'aticas, Facultad de Ciencias, Universidad de Zaragoza, Pedro cerbuna 12, 50009 Zaragoza (Spain), IUMA and Faculty of Computer Science and Mathematics, University of Passau, Innstrasse 33, 94032 Passau, Germany}
\email{j.martin@unizar.es; javier.martingoni@uni-passau.de}

\thanks{The first named author is partially supported by MICINN Project PID-105979-GB-I00 and DGA Project E48\_20R. The second named author is supported by DGA project E48\_20R and the Austrian Science Fund (FWF) Project P32405 \textit{Asymptotic Geometric Analysis and Applications.}}
\begin{abstract}
We consider the problem of finding the best function $\varphi_n:[0,1]\to\R$ such that for any pair of convex bodies $K,L\in\R^n$ the following Brunn-Minkowski type inequality holds
$$
|K+_\theta L|^\frac{1}{n}\geq\varphi_n(\theta)(|K|^\frac{1}{n}+|L|^\frac{1}{n}),
$$
where $K+_\theta L$ is the $\theta$-convolution body of $K$ and $L$. We prove a sharp inclusion of the family of Ball's bodies  of an $\alpha$-concave function in its super-level sets in order to provide the best possible function in the range $\left(\frac{3}{4}\right)^n\leq\theta\leq1$, characterizing the equality cases.
\end{abstract}

\date{\today}
\maketitle

\section{Introduction}

It is well known that for any pair of convex bodies (i.e., compact convex sets with non-empty interior) $K,L\subseteq\R^n$, their Minkowski sum $K+L$, defined as
$$
K+L:=\{x+y\,:x\in K, y\in L\}=\{z\in\R^n\,:\,K\cap(z-L)\neq\emptyset\},
$$
is a convex body whose volume (or $n$-dimensional Lebesgue measure) $|\cdot|$ verifies, by Brunn-Minkowski inequality (see \cite[Theorem 7.1.1]{Sch}), that
\begin{equation}\label{eq:BM}
|K+L|^\frac{1}{n}\geq|K|^\frac{1}{n}+|L|^\frac{1}{n}.
\end{equation}

In \cite{AJV}, the authors considered, for any $0\leq\theta\leq1$ the $\theta$-convolution bodies of a pair of convex bodies $K,L\subseteq\R^n$, defined as
$$
K+_\theta L:=\{z\in K+L\,:|K\cap (z-L)|\geq\theta M(K,L)\},
$$
where $M(K,L)=\max_{z\in\R^n}|K\cap(z-L)|$, and studied the problem of obtaining the best possible function $\varphi_n:[0,1]\to\R$ such that for any pair of convex bodies $K,L\subseteq\R^n$ one has the following Brunn-Minkowski type inequality
\begin{equation}\label{eq:BM_Theta}
|K+_\theta L|^\frac{1}{n}\geq\varphi_n(\theta)(|K|^\frac{1}{n}+|L|^\frac{1}{n}).
\end{equation}
The authors proved that for any pair of convex bodies $\frac{K+_\theta L}{1-\theta^\frac{1}{n}}$ is an increasing family of convex bodies in $\theta$ and, as a consequence of Brunn-Minkowski inequality, $$\varphi_n(\theta)\geq1-\theta^\frac{1}{n}$$ for every $\theta\in[0,1]$. Therefore, for every pair of convex bodies $K,L\subseteq\R^n$
\begin{equation}\label{BM_Convolution_Old}
|K+_\theta L|^\frac{1}{n}\geq(1-\theta^\frac{1}{n})(|K|^\frac{1}{n}+|L|^\frac{1}{n}).
\end{equation}
It was also shown in \cite{AJV} that the increasing sequence of convex bodies $\frac{K+_\theta L}{1-\theta^\frac{1}{n}}$ remains constant if and only if $K=-L$ is an $n$-dimensional simplex, in which case there is no equality in Brunn-Minkowski inequality \eqref{eq:BM}. The purpose of this paper is to improve the estimate of the function $\varphi_n$. We will prove the following:

\begin{thm}\label{thm:BM_Convolution_improved}
Let $K,L\subseteq\R^n$ be convex bodies. Then, for every $\left(\frac{3}{4}\right)^n\leq\theta\leq 1$ we have that
$$
|K+_\theta L|^\frac{1}{n}\geq \frac{1}{2}{{2n}\choose{n}}^\frac{1}{n}(1-\theta^\frac{1}{n})\left(|K|^\frac{1}{n}+|L|^\frac{1}{n}\right)
$$
and, for every  $0\leq\theta\leq \left(\frac{3}{4}\right)^n$,
$$
|K+_\theta L|^\frac{1}{n}\geq \left(1-\left(\frac{4}{3}-\frac{1}{6}{{2n}\choose{n}}^\frac{1}{n}\right)\theta^\frac{1}{n}\right)\left(|K|^\frac{1}{n}+|L|^\frac{1}{n}\right).
$$
Moreover, given any $\left(\frac{3}{4}\right)^n\leq\theta<1$ there is equality if and only if $K=-L$ is an $n$-dimensional simplex.
\end{thm}
\begin{rmk}
The constant $\frac{1}{2}{{2n}\choose{n}}^\frac{1}{n}$ is asymptotically $2$, so the result improves on \eqref{BM_Convolution_Old} for values of $\theta$ close to $1$, in which case there is equality whenever $K=-L$ is an $n$-dimensional simplex. Besides,
$\left(\frac{4}{3}-\frac{1}{6}{{2n}\choose{n}}^\frac{1}{n}\right)$ is asymptotically  $\frac{2}{3}$. Therefore, this result also improves on \eqref{BM_Convolution_Old}, for small values of $\theta$.
\end{rmk}

\begin{rmk}
The change of behavior in the estimate at $\theta_0=\left(\frac{3}{4}\right)^n$ is due to the method we use in order to prove Theorem \ref{thm:BM_Convolution_improved}. However, it is clear that the estimate obtained for the function $\varphi_n(\theta)$ for $\theta_0\leq\theta\leq1$ cannot hold in the whole range $0\leq\theta\leq1$, as this would lead, taking limit as $\theta$ tends to $0$, to Brunn-Minkowski's inequality \eqref{eq:BM} with an extra factor $\frac{1}{2}{{2n}\choose{n}}^\frac{1}{n}$ which is asymptotically $2$ and such inequality is false if $K=L$.
\end{rmk}

In order to prove Theorem \ref{thm:BM_Convolution_improved} we observe that the $\theta$-convolution bodies of two convex bodies $K,L\subseteq\R^n$ are the super-level sets of the function $\tilde{g}_{K,L}: \R^n\to[0,1]$ given by
$\tilde{g}_{K,L}(z)=\frac{|K\cap(z-L)|}{M(K,L)}$, which is $\frac{1}{n}$-concave (i.e., $\tilde{g}_{K,L}^\frac{1}{n}$ is concave on its support, $K+L$) and, in particular, is log-concave. Given any integrable log-concave function $g:\R^n\to[0,\infty)$ (i.e., $\log g:\R^n\to[-\infty,\infty)$ is concave) with $g(0)\neq0$, Ball \cite{Ba} defined the following family of convex bodies associated to it $(K_p(g))_{p>0}$:
$$
K_p(g):=\left\{x\in\R^n\,:\,p\int_0^\infty r^{p-1}g(rx)dr\geq g(0)\right\}.
$$
In \cite[Lemma 2.3.2, Remark 2.6]{ABG}, the authors proved, following the ideas in \cite{KM}, the following inclusion relation between Ball's bodies and super-level sets of a log-concave function in a certain range of the parameters involved: If $g:\R^n\to[0,\infty)$ is an integrable log-concave function with $\Vert g\Vert_\infty=g(0)$, then for any $p>0$ and  any $0<t<\frac{p}{e}$
\begin{equation}\label{eq:Inclusion_Log_Concave}
\frac{t}{\Gamma(1+p)^\frac{1}{p}}K_p(g)\subseteq \{x\in\R^n\,:\,g(x)=e^{-t}\Vert g\Vert_\infty\}.
\end{equation}
We will obtain a sharper inclusion relation between the family Ball's bodies and the super-level sets of an $\alpha$-concave function $g$ (i.e., $g^\alpha$ is concave on its support) which will allow us to prove Theorem \ref{thm:BM_Convolution_improved}. More precisely, denoting for any $0\leq r\leq 1$ by $L_r$ the super-level set of an $\alpha$-concave function $g$ given by
 \begin{eqnarray}\label{eq:SuperLevelSets}
L_r(g)&=&\{x\in\textrm{supp}(g)\,:\,g^\alpha(x)\geq r\Vert g\Vert_\infty^\alpha\}\cr
&=&\{x\in\textrm{supp}(g)\,:\,g(x)\geq r^{\frac{1}{\alpha}}\Vert g\Vert_\infty\},
\end{eqnarray}
and, denoting by $\partial K$  and by $\rho_K$ the boundary of the convex body $K\subseteq\R^n$ containing the origin in its interior and its radial function (see definition below), and for any $x,y>0$ the generalized binomial coefficient, which is defined in terms of the gamma function as $\displaystyle{ {x\choose y} := \frac{\Gamma{(1+x)}}{\Gamma{(1+y)}\Gamma{(1+x-y)}}}$, we will prove the following:
\begin{thm}\label{thm:BallsBodiesInclusion}
Let $p>0$ and let $K\subseteq\R^n$ be a convex body with $0\in K$ and let $g:K\to[0,\infty)$ be a continuous $\alpha$-concave function, with $\alpha>0$, such that $\Vert g\Vert_\infty=g(0)>0$ and such that $g(x)=0$ for every $x\in\partial K$. Then, understanding that $g(x)=0$ for every $x\not\in K$, we have that for every $t\in\left[0,\frac{p\alpha}{(1+p\alpha)^{1+\frac{1}{p\alpha}}}\right]$
$$
t{{p+\frac{1}{\alpha}}\choose{p}}^\frac{1}{p}K_p(g)\subseteq L_{1-t}(g).
$$
Moreover, for any $t\in \left(0,\frac{p\alpha}{(1+p\alpha)^{1+\frac{1}{p\alpha}}}\right]$ there is equality if and only if $g(x)=\Vert g\Vert_\infty\left(1-\Vert x\Vert_K\right)^\frac{1}{\alpha}$ for every $x\in K$. Furthermore, given $u\in S^{n-1}$, if there exists $r\in[0,\rho_K(u)]$ such that $g(ru)\neq\Vert g\Vert_\infty\left(1-\Vert ru\Vert_K\right)^\frac{1}{\alpha}$, we have that there exists $\varepsilon>0$ such that for every $t\in \left(0,\frac{p\alpha}{(1+p\alpha)^{1+\frac{1}{p\alpha}}}\right]$
$$
t{{p+\frac{1}{\alpha}}\choose{p}}^\frac{1}{p}(\rho_{K_p(g)}(u)+\varepsilon)\leq \rho_{L_{1-t}(g)}(u).
$$
\end{thm}

The paper is organised as follows. In Section \ref{sec:Preliminaries} we will provide the necessary definitions and results that we need to prove our results. In Section \ref{sec:Inclusion} we will provide the proof of Theorem \ref{thm:BallsBodiesInclusion}. Finally, in Section \ref{sec:BM_ThetaConvolution} we will prove Theorem \ref{thm:BM_Convolution_improved}.
\section{Preliminaries} \label{sec:Preliminaries}

\subsection{Notation}

Given a convex body $K\subseteq\R^n$ with $0\in\textrm{int}K$, we will denote by $\rho_K$ the radial function $\rho_K:S^{n-1}\to[0,\infty)$ given by $\rho_K(u)=\max\{\lambda\geq0\,:\,\lambda u\in K\}$, where $S^{n-1}$ denotes the $(n-1)$-dimensional Euclidean sphere in $\R^n$, that is, $S^{n-1}=\{u\in\R^n\,:\,\Vert u\Vert_2=1\}$. It is well known that  given two convex bodies $K_1,K_2\in\R^n$ we have that $K_1\subseteq K_2\Leftrightarrow\rho_{K_1}(u)\leq\rho_{K_2}(u)$ for every $u\in S^{n-1}$. $\Vert\cdot\Vert_K$ will denote the Minkowski gauge, defined as
$$
\Vert x\Vert_K=\inf\{\lambda>0\,:\,x\in\lambda K\}.
$$
Notice that for every $u\in S^{n-1}$ we have that $\Vert u\Vert_K=\frac{1}{\rho_K(u)}$. $\chi_K$ will denote the characteristic function of a convex body $K$  and $K-K$ will always denote the difference body $K+(-K)$. Whenever $K\subseteq\R^n$ is a $k$-dimensional set contained in an affine $k$-dimensional subspace, $|K|$ will denote its $k$-dimensional Lebesgue measure. $B_2^n$ will stand for the Euclidean unit ball, $\Vert\cdot\Vert_2$ for the Euclidean norm and $\Delta^n$ will stand for the regular simplex.
\subsection{Ball's bodies}

A log-concave function $g:\R^n\to[0,\infty)$ is a function of the form $g(x)=e^{-u(x)}$ with $u:\R^n\to(-\infty,\infty]$ a convex function. The family of log-concave functions plays an extremely important role in the study of problems related to distribution of volume in convex bodies. In \cite{Ba}, Ball introduced a family of convex bodies $(K_p(g))_{p>0}$ associated to log-concave functions verifying $g(0)>0$. More precisely, for any measurable (not necessarily log-concave) $g:\R^n\to[0,\infty)$ such that $g(0)>0$ and $p>0$, $K_p(g)$ is defined as
$$
K_p(g):=\left\{x\in\R^n\,:\,p\int_0^\infty r^{p-1}g(rx)dr\geq g(0)\right\}.
$$
$K_p(g)$ is a star set with center $0$ whose radial function is given by
$$
\rho_{K_p(g)}^p(u)=\frac{p}{g(0)}\int_0^\infty r^{p-1}g(ru)dr,\quad\forall u\in S^{n-1}.
$$
It is well known that for any integrable log-concave function with $g(0)>0$,  $K_p(g)$ is convex for every $p>0$ and, by integration in polar coordinates, we have that $\displaystyle{|K_n(g)|=\int_{\R^n}\frac{g(x)}{g(0)}dx}$.
We refer the reader to \cite[Section 2.5]{BGVV} for more information on Ball's bodies.

\subsection{The generalized covariogram function}\label{subsec:Covariogram}

Given a convex body $K\subseteq\R^n$, its covariogram function is the function $g_K:K-K\to[0,\infty)$ given by
$$
g_K(z)=|K\cap(z+K)|.
$$
Through this paper we will consider, given a pair of convex bodies $K,L\subseteq\R^n$, the generalized covariogram function defined as $g_{K,L}:K+L\to[0,\infty)$ given by
$$
g_{K,L}(z)=|K\cap(z-L)|=\chi_K*\chi_L(z),
$$
where $\chi_K*\chi_L$ denotes the convolution of the functions $\chi_K$ and $\chi_L$. Notice that for any convex body $g_K=g_{K,-K}$. As a consequence of Brunn-Minkowski inequality \eqref{eq:BM}, we have that for any pair of convex bodies $K,L\subseteq\R^n$, $g_{K,L}$ is $\frac{1}{n}$-concave and then, for any $\theta\in[0,1]$ we have that
$$
K+_\theta L=L_{\theta^\frac{1}{n}}(g_{K,L}),
$$
where $L_{\theta^\frac{1}{n}}(g_{K,L})$ is defined by \eqref{eq:SuperLevelSets}. Besides, (see \cite[Proposition 2.1]{AJV}) for any $x\in\R^n$ and any $\theta\in[0,1]$
\begin{equation}\label{eq:translation}
(x+K)+_\theta L=x+(K+_\theta L)
\end{equation}
and for any pair of convex bodies $K,L\subseteq\R^n$ one has that
$$
\int_{\R^n}\frac{g_{K,L}(z)}{\Vert g_{K,L}\Vert_\infty}dz=\frac{|K||L|}{M(K,L)}.
$$
It was also proved in \cite{AJV} that $\frac{K+_\theta L}{1-\theta^\frac{1}{n}}$ is an increasing family of convex bodies in $\theta\in[0,1]$ and (see \cite[Proposition 2.8]{AJV}) that $K+_\theta L=(1-\theta^\frac{1}{n})(K+L)$ for every $0\leq\theta\leq1$ if and only if $K=-L$ is a simplex.

\subsection{The polar projection body and Zhang's inequality}\label{sec:Zhang}

Given a convex body $K\subseteq\R^n$, its polar projection body $\Pi^*K$ is defined as the unit ball of the norm given by
$$
\Vert x\Vert_{\Pi^*K}=\Vert x\Vert_2|P_{x^\perp}K|.
$$
It is well known that for any convex body, $|K|^{n-1}|\Pi^*K|$ is an affine invariant quantity that verifies Petty projection inequality \cite{P} (also known as the affine isoperimetric inequality) :
$$
|K|^{n-1}|\Pi^*K|\leq|B_2^n|^{n-1}|\Pi^*B_2^n|,
$$
with equality if and only if $K$ is an ellipsoid. In \cite{Z}, Zhang proved the following reverse inequality for any convex body $K\subseteq\R^n$:
\begin{equation}\label{eq:ZhangsInequality}
|K|^{n-1}|\Pi^*K|\geq|\Delta^n|^{n-1}|\Pi^*\Delta^n|=\frac{1}{n^n}{{2n}\choose{n}},
\end{equation}
with equality if and only if $K$ is a simplex (see also \cite{GZ} for another proof).

In \cite{T}, Tsolomitis studied the existence of the behavior of limiting convolution bodies
$$
C_\alpha (K,L):=\lim_{\theta\to1^-}\frac{K+_\theta L}{(1-\theta)^\alpha}
$$
for symmetric convex bodies $K$ and $L$, and some exponent $\alpha$, giving some regularity conditions
under which the above limit is non-degenerated for some $\alpha$. Taking into account that for any convex body $K\subseteq\R^n$, $C_1(K,-K)=|K|\Pi^*K$,  in \cite[Theorem 4.6]{AJV}, the authors showed that Zhang's inequality can be extended to
\begin{equation}\label{eq:ZhangTwoBodies}
|C_1(K,L)|\geq\frac{1}{n^n}{{2n}\choose{n}}\frac{|K||L|}{M(K,L)},
\end{equation}
for any pair of convex bodies $K,L\subseteq\R^n$  such that $M(K,L)=|K\cap(-L)|$, with equality if and only if $K=-L$ is a simplex.

\section{An inclusion relation between Ball's bodies and superlevel sets}\label{sec:Inclusion}

In this section we are going to prove Theorem \ref{thm:BallsBodiesInclusion}, from which we will derive Theorem \ref{thm:BM_Convolution_improved} in the following section.

\begin{proof}[Proof of Theorem \ref{thm:BallsBodiesInclusion}]
Let us assume, without loss of generality that $\Vert g\Vert_\infty=g(0)=1$. Otherwise, consider the function $\frac{g}{\Vert g\Vert_\infty}$. Let us denote, for any $u\in S^{n-1}$, $l_u=\sup\{r>0\,:\,g(ru)>0\}=\rho_K(u)$, $v_u:[0,l_u]\to[0,1]$ the function defined as $v_u(r)=g^\alpha(ru)$, which is concave on $[0,l_u]$, and, for any $q>0$, let $\phi_u:[0,l_u]\to[0,\infty)$ the function defined as
$$
\phi_u(r)=r^{q\alpha}v_u(r)=r^{q\alpha}g^\alpha(ru).
$$
Notice that since $\log\phi_u$ is strictly concave on $(0,l_u)$, $\displaystyle{\lim_{r\to0^+}\log\phi_u(r)=-\infty}$, and $\displaystyle{\lim_{r\to l_u^-}\log\phi_u(r)=-\infty}$, $\phi_u$ attains a unique maximum at some $r_0=r_0(q)\in(0,l_u)$.
Therefore, denoting by $(\phi_u)_{-}(r_0)$ and $(\phi_u)_{+}(r_0)$ the lateral derivatives of $\phi_u$ at $r_0$ we have that
\begin{itemize}
\item $0\leq(\phi_u)_{-}(r_0)=r_0^{q\alpha}\left(\frac{q\alpha}{r_0}v_u(r_0)+(v_u)_{-}(r_0)\right)$,
\item $0\geq(\phi_u)_{+}(r_0)=r_0^{q\alpha}\left(\frac{q\alpha}{r_0}v_u(r_0)+(v_u)_{+}(r_0)\right)$.
\end{itemize}
Then,
\begin{itemize}
\item $(v_u)_{-}(r_0)\geq-\frac{q\alpha}{r_0}v_u(r_0)$,
\item $(v_u)_{+}(r_0)\leq-\frac{q\alpha}{r_0}v_u(r_0)$.
\end{itemize}
Notice that if $v_u$ is an affine function then necessarily for every $q>0$ we have that $v_u(r)=v_u(r_0)\left(1-\frac{q\alpha}{r_0}(r-r_0)\right)$. We will denote this affine function by
$$
\tau_{u,q}(r)=v_u(r_0(q))\left(1-\frac{q\alpha}{r_0(q)}(r-r_0(q))\right),\quad\forall r\in[0,l_u].
$$

For any $q>0$, the graph of the function $\tau_{u,q}$ is a supporting line of the hypograph of $v_u$ and then
$$
v_u(r)\leq \tau_{u,q}(r)=v_u(r_0)\left(1-\frac{q\alpha}{r_0}(r-r_0)\right),\quad\forall r\in[0,l_u].
$$
Thus, for any $p,q>0$
\begin{eqnarray}\label{eq:InequalityRhoKp}
\rho_{K_p(g)}^p(u)&=&p\int_0^{\infty}r^{p-1}g(ru)dr=p\int_0^{l_u}r^{p-1}v_u^\frac{1}{\alpha}(r)dr\cr
&\leq& pv_u^\frac{1}{\alpha}(r_0)\int_0^{l_u}r^{p-1}\left(1-\frac{q\alpha}{r_0}(r-r_0)\right)^\frac{1}{\alpha}dr\cr
&\leq&pg(r_0u)\int_0^{\left(1+\frac{1}{q\alpha}\right)r_0}r^{p-1}\left(1-\frac{q\alpha}{r_0}(r-r_0)\right)^\frac{1}{\alpha}dr\cr
&=&pg(r_0u)\left(1+q\alpha\right)^\frac{1}{\alpha}\left(1+\frac{1}{q\alpha}\right)^pr_0^p\int_0^1s^{p-1}\left(1-s\right)^\frac{1}{\alpha}ds\cr
&=&pg(r_0u)\frac{\left(1+q\alpha\right)^{p+\frac{1}{\alpha}}}{(q\alpha)^p}r_0^p\beta\left(p,1+\frac{1}{\alpha}\right).
\end{eqnarray}
Moreover, for any $q>0$, the previous inequality is an equality if and only if $l_u=\left(1+\frac{1}{q\alpha}\right)r_0$ and  $v_u(r)=\tau_{u,q}(r)$ for every $0\leq r\leq l_u=\left(1+\frac{1}{q\alpha}\right)r_0$. That is, if $v_u$ is an affine function such that $v_u(l_u)=0$.

Consequently, for any $p,q>0$
$$
\frac{q\alpha}{\left(1+q\alpha\right)^{1+\frac{1}{p\alpha}}}{{p+\frac{1}{\alpha}}\choose{p}}^\frac{1}{p}\rho_{K_p(g)}(u)\leq g(r_0u)^\frac{1}{p}r_0,
$$
with equality if and only if $v_u$ is an affine function such that $v_u(l_u)=0$.

On the one hand, since $v_u$ is concave on $[0, l_u]$ and $\Vert g\Vert_\infty=g(0)=1$, we have that
\begin{eqnarray*}
g^\alpha\left(g^\frac{1}{p}(r_0u)r_0u\right)&=&v_u\left(g^\frac{1}{p}(r_0u)r_0\right)\geq  g^\frac{1}{p}(r_0u)v_u(r_0)+\left(1- g^\frac{1}{p}(r_0u)\right)v_u(0)\cr
&=&g^\frac{1}{p}(r_0u)g^\alpha(r_0u)+1- g^\frac{1}{p}(r_0u)\cr
&=&1-\left(g^\alpha(r_0u)\right)^\frac{1}{p\alpha}\left(1-g^\alpha(r_0u)\right),
\end{eqnarray*}
with equality if and only if $v_u$ is affine on $[0,r_0]$ and then $v_u(r)=\tau_{u,q}(r)$ for every $0\leq r\leq r_0$. On the other hand, since $v_u$ is concave on $[0, l_u]$, we have that $(v_u)_-$ is decreasing on $[0, l_u]$ and then
\begin{eqnarray*}
v_u(r_0)&=&v_u(0)+\int_0^{r_0}(v_u)_-(r)dr\geq v_u(0)+(v_u)_-(r_0)r_0\geq v_u(0)-q\alpha v_u(r_0)\cr
&=&1-q\alpha v_u(r_0)
\end{eqnarray*}
and then
$$
g^\alpha(r_0u)=v_u(r_0)\geq\frac{1}{1+q\alpha},
$$
with equality if and only if $(v_u)_-=-\frac{q\alpha}{r_0} v_u(r_0)$ for every $0\leq r<r_0$, in which case $v_u(r)=\tau_{u,q}(r)$ for every $0\leq r< r_0$. Since the function $h(x)=1-x^\frac{1}{p\alpha}(1-x)$ is decreasing in $\left[0,\frac{1}{1+p\alpha}\right]$ and is increasing in $\left[\frac{1}{1+p\alpha},1\right]$, we have that if $\frac{1}{1+p\alpha}\leq\frac{1}{1+q\alpha}$ (which happens whenever $0<q\leq p$)
\begin{eqnarray*}
g^\alpha\left(g^\frac{1}{p}(r_0u)r_0u\right)&\geq&1-\left(g^\alpha(r_0u)\right)^\frac{1}{p\alpha}\left(1-g^\alpha(r_0u)\right)=h(g^\alpha(r_0u))\geq h\left(\frac{1}{1+q\alpha}\right)\cr
&=&1-\left(\frac{1}{1+q\alpha}\right)^\frac{1}{p\alpha}\left(1-\frac{1}{1+q\alpha}\right)\cr
&=&1-\frac{q\alpha}{(1+q\alpha)^{1+\frac{1}{p\alpha}}}.
\end{eqnarray*}
Besides, there is equality if and only if $v_u(r)=\tau_{u,q}(r)$ for every $0\leq r\leq r_0$ and
$$
g^\alpha(r_0u)=v_u(r_0)=\frac{1}{1+q\alpha},
$$
which also happens if and only if $v_u(r)=\tau_{u,q}(r)$ for every $0\leq r< r_0$.

Therefore, if $0<q\leq p$, $g^\frac{1}{p}(r_0u)r_0u\in L_{1-\frac{q\alpha}{(1+q\alpha)^{1+\frac{1}{p\alpha}}}}(g)$ and, since $0\in L_{1-\frac{q\alpha}{(1+q\alpha)^{1+\frac{1}{p\alpha}}}}(g)$ as $\Vert g\Vert_\infty=g(0)=1$ and $L_{1-\frac{q\alpha}{(1+q\alpha)^{1+\frac{1}{p\alpha}}}}(g)$ is convex, we have that
$$g^\frac{1}{p}(r_0u)r_0\leq \rho_{L_{1-\frac{q\alpha}{(1+q\alpha)^{1+\frac{1}{p\alpha}}}}}(u),
$$
with equality if and only if $g^\frac{1}{p}(r_0u)r_0u\in\partial L_{1-\frac{q\alpha}{(1+q\alpha)^{1+\frac{1}{p\alpha}}}}$, which happens if and only if $v_u(r)=\tau_{u,q}(r)$ for every $0\leq r\leq r_0$. Thus, if $0<q\leq p$
$$
\frac{q\alpha}{\left(1+q\alpha\right)^{1+\frac{1}{p\alpha}}}{{p+\frac{1}{\alpha}}\choose{p}}^\frac{1}{p}\rho_{K_p(g)}(u)\leq \rho_{L_{1-\frac{q\alpha}{(1+q\alpha)^{1+\frac{1}{p\alpha}}}}}(u),
$$
with equality if and only  $v_u(r)=\tau_{u,q}(r)$ for every $0\leq r\leq l_u=\left(1+\frac{1}{q\alpha}\right)r_0$, i.e., if $v_u$ is an affine function such that $v_u(l_u)=0$. Since this happens for every $u\in S^{n-1}$,
$$
\frac{q\alpha}{\left(1+q\alpha\right)^{1+\frac{1}{p\alpha}}}{{p+\frac{1}{\alpha}}\choose{p}}^\frac{1}{p}K_p(g)\subseteq L_{1-\frac{q\alpha}{(1+q\alpha)^{1+\frac{1}{p\alpha}}}}(g)
$$
and, for any $0<q\leq p$, there is equality equality if and only if for every direction $u\in S^{n-1}$ $v_u$ is an affine function such that $v_u(l_u)=0$. That is, if $g(x)=\Vert g\Vert_\infty\left(1-\Vert x\Vert_K\right)^\frac{1}{\alpha}$ for every $x\in K$. Since the function $h_1(x)=\frac{x}{(1+x)^{1+\frac{1}{p\alpha}}}$ is continuous and increasing in $[0,p\alpha]$, it attains every value in $\left[0,\frac{p\alpha}{(1+p\alpha)^{1+\frac{1}{p\alpha}}}\right]$. Therefore, we have that for every $t\in\left[0,\frac{p\alpha}{(1+p\alpha)^{1+\frac{1}{p\alpha}}}\right]$.
$$
t{{p+\frac{1}{\alpha}}\choose{p}}^\frac{1}{p}K_p(g)\subseteq L_{1-t}(g)
$$
and, for any $t\in\left[0,\frac{p\alpha}{(1+p\alpha)^{1+\frac{1}{p\alpha}}}\right]$, there is equality if and only if $g(x)=\Vert g\Vert_\infty\left(1-\Vert x\Vert_K\right)^\frac{1}{\alpha}$.

Assume now that for some $u\in S^{n-1}$ there exists $r\in[0,\rho_K(u)]$ such that $g(ru)\neq\Vert g\Vert_\infty\left(1-\Vert ru\Vert_K\right)^\frac{1}{\alpha}$. Therefore, $v_u$ is not an affine function.

We first notice that the function $r_0(q)$ is continuous on $q\in (0,\infty)$. Indeed, let $(q_k)_{k=1}^\infty\subseteq(0,\infty)$ be a sequence converging to some $q\in(0,\infty)$. Let  $(q_{k_i})_{i=1}^\infty$ be any convergent subsequence of $(q_k)_{k=1}^\infty$ such that $r_0(q_{k_i})$ converges to some $\overline{r}\in[0,l_u]$, which exist since $(r_0(q_k))_{k=1}^\infty\subseteq(0,l_u)$. Since for every $i\in\N$, we have, by the definition of $r_0(q_{k_i})$, that
$$
r_0(q)^{q_{k_i}}v_u(r_0(q))\leq r_0(q_{k_i})^{q_{k_i}}v_u(r_0(q_{k_i})),
$$
taking limits as $i$ tends to $\infty$ we obtain that
$$
r_0(q)^{q}v_u(r_0(q))\leq \overline{r}^{q}v_u(\overline{r}).
$$
Therefore, by the definition of $r_0(q)$, $\overline{r}=r_0(q)$. Therefore,
$$
\liminf_{k\to\infty}r_0(q_k)=\limsup_{k\to\infty}r_0(q_k)=r_0(q)
$$
and then $r_0(q_k)$ converges to $r_0(q)$. Thus $r_0(q)$ is continuous on $q\in(0,\infty)$. Besides, if $(q_k)_{k=1}^\infty\subseteq(0,\infty)$ is a sequence converging to $0$ and for some subsequence $r_0(q_{k_i})$ converges to $l_u$ we would have that for every $r\in [0,l_u]$
$$
v_u(r)\leq \tau_{u,q_{k_i}}(r)=v_u(r_0(q_{k_i}))\left(1-\frac{q_{k_i}\alpha}{r_0(q_{k_i})}(r-r_0(q_{k_i}))\right),
$$
leading to $v_u(r)\leq 0$, which is a contradiction. Therefore, for any $p>0$ we have that $s:=\sup\{r_0(q)\,:\,q\in(0,p]\}<l_u$ and then, for every $p>0$ and every $0<q\leq p$ we have that $\frac{q\alpha}{r_0(q)}v_u(r_0(q))\leq-(v_u)_+(r_0(q))\leq -(v_u)_+(s)$ and then $\frac{q\alpha}{r_0(q)}v_u(r_0(q))$ is bounded in $q\in(0,p]$.

Assume that there is no $\varepsilon>0$ such that for every $0<q\leq p$ we have that
$$
\varepsilon<p\int_0^{\left(1+\frac{1}{q\alpha}\right)r_0(q)}r^{p-1}\tau_{u,q}^\frac{1}{\alpha}(r)dr-p\int_0^{l_u}r^{p-1}v_u^\frac{1}{\alpha}(r)dr.
$$
Then, we can find a sequence $(q_k)_{k=1}^\infty\subseteq(0,p]$ and, if necessary, extract from it further subsequences which we denote in the same way, such that
\begin{eqnarray*}
\lim_{k\to\infty}p\int_0^{\infty}r^{p-1}\left(\tau_{u,q_k}^\frac{1}{\alpha}(r)\chi_{\left(1+\frac{1}{q_k\alpha}\right)r_0(q_k)}(r)-v_u^\frac{1}{\alpha}(r)\chi_{[0,l_u]}(r)\right)dr=0,
\end{eqnarray*}
$(q_k)_{k=1}^\infty$ converges to some $q\in[0,p]$, $r_0(q_k)$ converges to some $\overline{r}\in[0,l_u]$, and $\frac{q_k\alpha}{r_0(q_k)}v_u(r_0)$ converges to some $\lambda\in[0,\infty)$. Since for every $r\in[0,\infty)$ we have that for every $k\in\N$
\begin{eqnarray*}
v_u(r)\chi_{[0,l_u]}(r)&\leq&\tau_{u,q_k}(r)\chi_{\left[0,\left(1+\frac{1}{q_k\alpha}\right)r_0(q_k)\right]}(r)\cr
&=&v_u(r_0(q_k))\left(1-\frac{q_k\alpha}{r_0(q_k)}(r-r_0(q_k))\right)\chi_{\left[0,\left(1+\frac{1}{q_k\alpha}\right)r_0(q_k)\right]}(r),
\end{eqnarray*}
taking limits as $k\to\infty$ we obtain that for almost every $r\in[0,\infty)$
$$
v_u(r)\chi_{[0,l_u]}(r)\leq \left(v_u(\overline{r})-\lambda (r-\overline{r})\right)\chi_{\left[0,\left(1+\frac{1}{q\alpha}\right)\overline{r}\right]}(r)
$$
and then, by Fatou's lemma,
\begin{align*}
& p\int_0^\infty r^{p-1}\left(\left(v_u(\overline{r})-\lambda (r-\overline{r})\right)\chi_{\left[0,\left(1+\frac{1}{q\alpha}\right)^\frac{1}{\alpha}\overline{r}\right]}(r)-v_u(r)^\frac{1}{\alpha}\chi_{[0,l_u]}(r)\right)dr\cr
&\leq\lim_{k\to\infty}p\int_0^{\infty}r^{p-1}\left(\tau_{u,q_k}^\frac{1}{\alpha}(r)\chi_{\left[0,\left(1+\frac{1}{q_k\alpha}\right)r_0(q_k)\right]}(r)-v_u^\frac{1}{\alpha}(r)\chi_{[0,l_u]}(r)\right)dr=0.
\end{align*}
Since the integrand in the first integral is non-negative, by continuity of the following functions, we have that for every  $r\in[0,\infty)$
$$
v_u(r)\chi_{[0,l_u]}(r)= \left(v_u(\overline{r})-\lambda (r-\overline{r})\right)\chi_{\left[0,\left(1+\frac{1}{q\alpha}\right)\overline{r}\right]}(r)
$$
and then $v_u$ is an affine function. Therefore, if $v_u$ is not linear, there exists $\overline{\varepsilon}>0$ such that for every $0<q\leq p$
$$
p\int_0^{l_u}r^{p-1}v_u^\frac{1}{\alpha}(r)dr+\overline{\varepsilon}\leq pg(r_0u)\frac{\left(1+q\alpha\right)^{p+\frac{1}{\alpha}}}{(q\alpha)^p}r_0^p\beta\left(p,1+\frac{1}{\alpha}\right)
$$
and then there exists $\varepsilon>0$ such that for every $0<q\leq p$
\begin{eqnarray*}
\frac{q\alpha}{\left(1+q\alpha\right)^{1+\frac{1}{p\alpha}}}{{p+\frac{1}{\alpha}}\choose{p}}^\frac{1}{p}\left(\rho_{K_p(g)}(u)+\varepsilon\right)&\leq&\frac{q\alpha}{\left(1+q\alpha\right)^{1+\frac{1}{p\alpha}}}{{p+\frac{1}{\alpha}}\choose{p}}^\frac{1}{p}\left(\rho_{K_p(g)}^p(u)+\overline{\varepsilon}\right)^\frac{1}{p}\cr
&\leq& g(r_0u)^\frac{1}{p}r_0,
\end{eqnarray*}
Proceeding now as in the proof of the inequality we obtain the result.
\end{proof}

\section{Brunn-Minkowski inequality for $\theta$-convolution bodies}\label{sec:BM_ThetaConvolution}

In this section we are going to prove Theorem \ref{thm:BM_Convolution_improved}.

\begin{proof}[Proof of Theorem \ref{thm:BM_Convolution_improved}]
Let $K,L\subseteq\R^n$ be a pair of convex bodies. By \eqref{eq:translation} we can assume, without loss of generality, that
$$
M(K,L)=\max_{z\in\R^n}|K\cap(z-L)|=|K\cap (-L)|.
$$
Then, the function $g_{K,L}:K+L\to[0,\infty)$ given by $g_{K,L}(z)=|K\cap(z-L)|$, which is a continuous $\frac{1}{n}$-concave function, verifies that $\Vert g_{K,L}\Vert_\infty=g_{K,L}(0)$ and $g_{K,L}(z)=0$ for every $z\in\partial(K+L)$.

By Theorem \ref{thm:BallsBodiesInclusion} with $p=n$, we have that for every $0\leq t\leq\frac{1}{4}$
$$
t{{2n}\choose{n}}^\frac{1}{n}K_n(g_{K,L})\subseteq L_{1-t}(g_{K,L}).
$$
Equivalently, taking $t=1-\theta^\frac{1}{n}$ we have that if $\left(\frac{3}{4}\right)^n\leq\theta<1$
$$
{{2n}\choose{n}}^\frac{1}{n}K_n(g_{K,L})\subseteq \frac{L_{\theta^\frac{1}{n}}(g_{K,L})}{1-\theta^\frac{1}{n}}=\frac{K+_\theta L}{1-\theta^\frac{1}{n}}.
$$
Taking volumes, and taking into account that
$$
|K_n(g_{K,L})|=\int_{\R^n}\frac{g_{K,L}(x)}{g_{K,L}(0)}dx=\int_{\R^n}\frac{ g_{K,L}(x)}{\Vert g_{K,L}\Vert_\infty}dx=\frac{|K||L|}{M(K,L)}
$$
we obtain that for every $\left(\frac{3}{4}\right)^n\leq\theta<1$
\begin{eqnarray*}
\left|\frac{K+_\theta L}{1-\theta^\frac{1}{n}}\right|^\frac{1}{n}&\geq&{{2n}\choose{n}}^\frac{1}{n}|K_n(g_{K,L})|^\frac{1}{n}={{2n}\choose{n}}^\frac{1}{n}\frac{|K|^\frac{1}{n}|L|^\frac{1}{n}}{M(K,L)^\frac{1}{n}}\geq{{2n}\choose{n}}^\frac{1}{n}\frac{|K|^\frac{1}{n}|L|^\frac{1}{n}}{\min\{|K|^\frac{1}{n},|L|^\frac{1}{n}\}}\cr
&=&{{2n}\choose{n}}^\frac{1}{n}\max\{|K|^\frac{1}{n},|L|^\frac{1}{n}\}\geq\frac{1}{2}{{2n}\choose{n}}^\frac{1}{n}(|K|^\frac{1}{n}+|L|^\frac{1}{n}).
\end{eqnarray*}
Assume that there is equality for some $\left(\frac{3}{4}\right)^n\leq\theta_0<1$. Then, by the equality cases in Theorem \ref{thm:BallsBodiesInclusion}, we have that
$$
g_{K,L}(x)=M(K,L)\left(1-\Vert x\Vert_{K+L}\right)^n\quad\forall x\in K+L.
$$
Otherwise, we have that for every $0\leq t\leq\frac{1}{4}$
$$
t{{2n}\choose{n}}^\frac{1}{n}K_n(g_{K,L})\subsetneq L_{1-t}(g_{K,L}).
$$
or, equivalently, taking $t=1-\theta^\frac{1}{n}$ we have that for every $\left(\frac{3}{4}\right)^n\leq\theta<1$
$$
{{2n}\choose{n}}^\frac{1}{n}K_n(g_{K,L})\subsetneq \frac{L_{\theta^\frac{1}{n}}(g_{K,L})}{1-\theta^\frac{1}{n}}=\frac{K+_\theta L}{1-\theta^\frac{1}{n}}.
$$
and then, in particular,
$$
\left|\frac{K+_{\theta_0} L}{1-\theta_0^\frac{1}{n}}\right|^\frac{1}{n}>{{2n}\choose{n}}^\frac{1}{n}|K_n(g_{K,L})|^\frac{1}{n}\geq\frac{1}{2}{{2n}\choose{n}}^\frac{1}{n}(|K|^\frac{1}{n}+|L|^\frac{1}{n}),
$$
which contradicts the equality at $\theta_0$.

Therefore, if there is equality for some $\left(\frac{3}{4}\right)^n\leq\theta_0\leq1$, we have that for every $0\leq\theta<1$
$$
K+_\theta L=(1-\theta^\frac{1}{n})(K+L)
$$
and then, as mentioned in Section \ref{subsec:Covariogram}, $K=-L$ is a simplex.

For any $0\leq\theta\leq \left(\frac{3}{4}\right)^n$ we have that $0\leq\theta^\frac{1}{n}\leq\frac{3}{4}$ and
$$
\theta^\frac{1}{n}=\left(\frac{4}{3}\theta^\frac{1}{n}\right)\frac{3}{4}+\left(1-\frac{4}{3}\theta^\frac{1}{n}\right)0.
$$
Since $g_{K,L}^\frac{1}{n}$ is concave on $K+L$, we have that
\begin{eqnarray*}
K+_\theta L&=&L_{\theta^\frac{1}{n}}(g_{K,L})\supseteq\frac{4}{3}\theta^\frac{1}{n}L_{\frac{3}{4}}(g_{K,L})+\left(1-\frac{4}{3}\theta^\frac{1}{n}\right)L_0(g_{K,L})\cr
&\supseteq&\frac{1}{3}\theta^\frac{1}{n}{{2n}\choose{n}}^\frac{1}{n}K_n(g_{K,L})+\left(1-\frac{4}{3}\theta^\frac{1}{n}\right)(K+L),
\end{eqnarray*}
where the last inclusion relation is a consequence of Theorem \ref{thm:BallsBodiesInclusion}. Taking volumes and using Brunn-Minkowski inequality we obtain that
\begin{eqnarray*}
|K+_\theta L|^\frac{1}{n}&\geq& \frac{1}{3}\theta^\frac{1}{n}{{2n}\choose{n}}^\frac{1}{n}|K_n(g_{K,L})|^\frac{1}{n}+\left(1-\frac{4}{3}\theta^\frac{1}{n}\right)|K+L|^\frac{1}{n}\cr
&\geq&\frac{1}{3}\theta^\frac{1}{n}{{2n}\choose{n}}^\frac{1}{n}\frac{|K|^\frac{1}{n}|L|^\frac{1}{n}}{M(K,L)^\frac{1}{n}}+\left(1-\frac{4}{3}\theta^\frac{1}{n}\right)\left(|K|^\frac{1}{n}+|L|^\frac{1}{n}\right)\cr
&\geq&\frac{1}{3}\theta^\frac{1}{n}{{2n}\choose{n}}^\frac{1}{n}\max\{|K|^\frac{1}{n},|L|^\frac{1}{n}\}+\left(1-\frac{4}{3}\theta^\frac{1}{n}\right)\left(|K|^\frac{1}{n}+|L|^\frac{1}{n}\right)\cr
&\geq&\frac{1}{6}\theta^\frac{1}{n}{{2n}\choose{n}}^\frac{1}{n}\left(|K|^\frac{1}{n}+|L|^\frac{1}{n}\right)+\left(1-\frac{4}{3}\theta^\frac{1}{n}\right)\left(|K|^\frac{1}{n}+|L|^\frac{1}{n}\right)\cr
&=&\left(1-\left(\frac{4}{3}-\frac{1}{6}{{2n}\choose{n}}^\frac{1}{n}\right)\theta^\frac{1}{n}\right)\left(|K|^\frac{1}{n}+|L|^\frac{1}{n}\right).
\end{eqnarray*}
\end{proof}

Finally, let us point out that, as a consequence of the estimates obtained in the proof of Theorem \ref{thm:BM_Convolution_improved}, we can obtain another proof of inequality \eqref{eq:ZhangTwoBodies}:
\begin{cor}
Let $K,L\subseteq\R^n$ be a pair of convex bodies such that $M(K,L)=|K\cap(-L)|$. Then
$$
|C_1(K,L)|^\frac{1}{n}\geq\frac{1}{n}{{2n}\choose{n}}^\frac{1}{n}\frac{|K|^\frac{1}{n}|L|^\frac{1}{n}}{M(K,L)^\frac{1}{n}},
$$
with equality if and only if $K=-L$ is a simplex.
\end{cor}

\begin{proof}
We have seen in the previous proof that for every $\left(\frac{3}{4}\right)^n\leq\theta<1$
$$
\left|\frac{K+_\theta L}{1-\theta^\frac{1}{n}}\right|^\frac{1}{n}\geq{{2n}\choose{n}}^\frac{1}{n}\frac{|K|^\frac{1}{n}|L|^\frac{1}{n}}{M(K,L)^\frac{1}{n}}.
$$
Taking limit as $\theta\to1^-$ and taking into account that $\frac{K+_\theta L}{1-\theta^\frac{1}{n}}$ is an increasing family of convex bodies in $\theta$ such that
$$
\lim_{\theta\to1^-}\frac{K+_\theta L}{1-\theta^\frac{1}{n}}=\lim_{\theta\to1^-}\frac{1-\theta}{1-\theta^\frac{1}{n}}\frac{K+_\theta L}{1-\theta}=nC_1(K,L),
$$
we obtain the result. Moreover, if $K=-L$ is a simplex, the inequality above becomes Zhang's inequality \eqref{eq:ZhangsInequality}. On the other hand, if there is equality, by the equality case in Theorem \ref{thm:BallsBodiesInclusion}, then necessarily $g_{K,L}=M(K,L)(1-\Vert x\Vert_{K+L})^n$. Therefore, for every $\theta\in[0,1]$ we have that
$$
K+_\theta L=(1-\theta^\frac{1}{n})(K+L)
$$
and then $K=-L$ is a simplex.
\end{proof}

\end{document}